\documentclass[a4paper,11pt]{article}
\usepackage[utf8]{inputenc}
\usepackage{pstricks,pst-tree,pst-node} 
\usepackage{subcaption}
\usepackage{amsmath}
\usepackage{amsfonts}
\usepackage{amssymb}
\usepackage{amsthm}
\usepackage{mathrsfs}
\usepackage{amsthm}
\usepackage{enumitem}
\usepackage[left=3cm, right=3cm,top=3cm,bottom=3cm]{geometry}
\usepackage{graphicx}

\usepackage{pstricks}
\theoremstyle{plain}
\newtheorem{theorem}{Theorem}

\newtheorem{lemma}[theorem]{Lemma}
\newtheorem{corollary}[theorem]{Corollary}
\newtheorem{conjecture}[theorem]{Conjecture}
\newtheorem{definition}[theorem]{Definition}

\theoremstyle{definition}
\newtheorem*{ack}{Acknowledgements}
\title{Counterexamples to Thomassen's conjecture on decomposition of cubic graphs}
\author{
  Thomas Bellitto \footnote{Department of Mathematics and Computer Science, University of Southern Denmark, Odense, Denmark. \texttt{bellitto@imada.sdu.dk}}
  \and
  Tereza Klimošová \footnote{Department of Applied Mathematics, Faculty of Mathematics and Physics, Charles University, Prague, Czech Republic \texttt{tereza@kam.mff.cuni.cz}}
  \and
  Martin Merker \footnote{Department of Applied Mathematics and Computer Science, Technical University of Denmark, Lyngby, Denmark. \texttt{marmer@dtu.dk}}
  \and 
  Marcin Witkowski \footnote{Faculty of Mathematics and Computer Science, Adam Mickiewicz University, Pozn\'{a}n, Poland \texttt{mw@amu.edu.pl}}
  \and 
  Yelena Yuditsky\footnote{Department of Mathematics, Ben-Gurion University of the Negev, Be’er-Sheva, Israel. \texttt{yuditsky@bgu.ac.il}}
}
\begin{document}
 \maketitle

\begin{abstract}
We construct an infinite family of counterexamples to Thomassen's conjecture that the vertices of every 3-connected, cubic graph on at least 8 vertices can be colored blue and red such that the blue subgraph has maximum degree at most 1 and the red subgraph minimum degree at least 1 and contains no path on 4 vertices. 
\end{abstract}
 
Wegner~\cite{bib:wegner} conjectured in 1977 that the square of every planar, cubic graph is 7-colorable and this was recently proved by Thomassen~\cite{bib:thomassen} and independently by Hartke, Jahanbekam and Thomas~\cite{bib:hartke}. The general idea of Thomassen's proof is that a special 2-coloring of the vertices of a cubic graph can be used to obtain a 7-coloring of its square. In this article, we call such a 2-coloring good, which is defined as follows.

\begin{definition}
A \emph{good coloring} of a graph is a 2-coloring of its vertices in colors blue and red such that 
\begin{description}[topsep = 0pt, noitemsep]
    \item[(1)] the subgraph induced by the blue vertices has maximum degree at most 1,
    \item[(2)] the the subgraph induced by the red vertices has minimum degree at least 1, and 
    \item[(3)] the subgraph induced by the red vertices contains no path on 4 vertices.
\end{description}
\end{definition}

It is easy to prove that a minimal counterexample to Wegner's conjecture would have to be cubic and 3-connected (see Lemma 3 of~\cite{bib:hartke}) and would of course have at least 8 vertices.
Thomassen showed in~\cite{bib:thomassen} that if a 3-connected planar cubic graph has a good coloring, then its square is 7-colorable. Hence, Thomassen made the following conjecture that could lead to a substantially simpler proof of Wegner's conjecture.:

\begin{conjecture}[Thomassen]\label{conj:thomassen}
Every 3-connected, cubic graph on at least 8 vertices has a good coloring.
\end{conjecture}

Note that the restriction to graphs on at least 8 vertices is necessary to exclude the 3-prism which does not have a good coloring. Bar\'at~\cite{bib:barat} proved that Conjecture~\ref{conj:thomassen} holds for generalized Petersen graphs. He also showed that every subcubic tree admits a good coloring. This motivated him to propose the following strengthening of Conjecture~\ref{conj:barat}.

\begin{conjecture}[Bar\'at]\label{conj:barat}
Every subcubic graph on at least 7 vertices has a good coloring.
\end{conjecture}

We construct an infinite family of counterexamples to Conjecture~\ref{conj:thomassen} which also disproves Conjecture~\ref{conj:barat}. The gadgets of our construction are defined as follows.  

\begin{definition}[$H$, $H'$, $H''$]
Let $H$ be the graph consisting of an 8-cycle $v_0v_1\ldots v_7$ with two chords $v_2v_6$ and $v_3v_7$, see Figure~\ref{figH}.\\
Let $H'$ be the graph consisting of two disjoint copies of $H$ and two edges joining the two copies as in Figure~\ref{figH'}.\\
Let $H''$ be the graph consisting of three disjoint copies of $H'$ and three edges joining the copies of $H'$ as in Figure~\ref{figH''}.
\end{definition}

\begin{figure}[!h]
\centering
\begin{minipage}{0.4\textwidth}
\includegraphics[scale = 0.8]{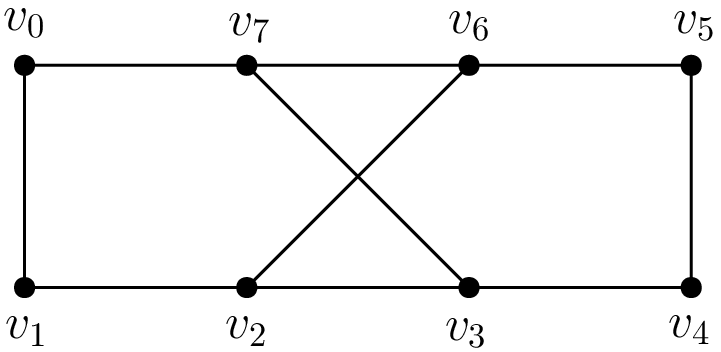}
\caption{The graph $H$}
\label{figH}
\end{minipage}
\hfill
\begin{minipage}{0.5\textwidth}
\includegraphics[scale = 0.8]{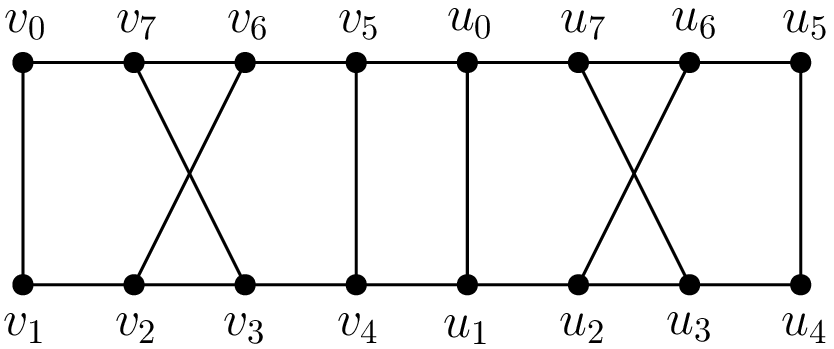}
\caption{The graph $H'$}
\label{figH'}
\end{minipage}
\end{figure}

Our goal is to show that every subcubic graph containing $H''$ as a subgraph has no good coloring. Note that if $H$ is a subgraph of a subcubic graph $G$, only the vertices that have degree 2 in $H$ can have a neighbor in $G-H$. The same applies for $H'$ and $H''$.
In the following we use the notation from Figures~\ref{figH} and~\ref{figH'} to refer to the vertices of $H$ and $H'$.

\begin{lemma}\label{lemmaH} 
If $H$ is an induced subgraph of a subcubic graph $G$, then in every good coloring of $G$
\begin{itemize}[topsep = 0pt, noitemsep]
    \item at most one of the vertices $v_0,v_1$ is colored red, and
    \item if one of $v_0,v_1$ is colored red and its neighbor in $H$ is also colored red, then both $v_4$ and $v_5$ are colored blue.
\end{itemize}
\end{lemma}

\begin{proof}
By contradiction, suppose that both $v_0$ and $v_1$ are colored red in a good coloring of $G$.
If one of $v_2$ and $v_7$, say $v_2$, is also colored red, then the vertices $v_3$, $v_6$ and $v_7$ are all colored blue by (3). However, now $v_7$ is blue and has two blue neighbors, contradicting (1). Thus we may assume that both $v_2$ and $v_7$ are colored blue. By (1), both $v_3$ and $v_6$ are colored red. By (2), $v_3$ and $v_6$ each need a red neighbour, so also $v_4$ and $v_5$ are colored red. Now $v_3v_4v_5v_6$ is a red path on 4 vertices, contradicting (3).\\
To prove the second part of the lemma, we may assume that $v_0$ is colored blue and $v_1$, $v_2$ are colored red. If $v_3$ is colored blue, then $v_7$ is colored red by (1). By (2), $v_6$ is colored red. Now $v_1v_2v_6v_7$ is a red path on 4 vertices, contradicting (3). Thus we may assume that $v_3$ is colored red. By (3), both $v_7$ and $v_4$ are colored blue. By (1), $v_6$ is colored red. Finally, by (3), $v_5$ is colored blue, so both $v_4$ and $v_5$ are colored blue.
\qedhere
\end{proof}

\begin{lemma}\label{lemmaH'}
If $H'$ is an induced subgraph of a subcubic graph $G$, then in a good coloring of $G$ exactly one of the following statements is true:
\begin{itemize}[topsep = 0pt, noitemsep]
    \item both $v_0$, $v_1$ are colored blue, or
    \item $v_0$ is colored red, $v_1$ is colored blue, and $v_0$ has a red neighbor in $G-H'$, or
    \item $v_1$ is colored red, $v_0$ is colored blue, and $v_1$ has a red neighbor in $G-H'$.
\end{itemize}
\end{lemma}
\begin{proof}
We may assume that not both $v_0$ and $v_1$ are colored blue. By Lemma~\ref{lemmaH} not both $v_0$ and $v_1$ can be colored red. Thus, we may assume that $v_0$ is red and $v_1$ is blue. Suppose for a contradiction that $v_0$ has no red neighbor in $G-H'$. By (2) and Lemma~\ref{lemmaH}, both $v_4$ and $v_5$ are colored blue. Thus, by (1), both $u_0$ and $u_1$ are colored red. However, the vertices $u_0u_1\ldots u_7$ induce a copy of $H$, so by Lemma~\ref{lemmaH} at most one of $u_0$, $u_1$ can be red.
\end{proof}

\begin{figure}[!h]
\centering
\includegraphics[scale = 0.5]{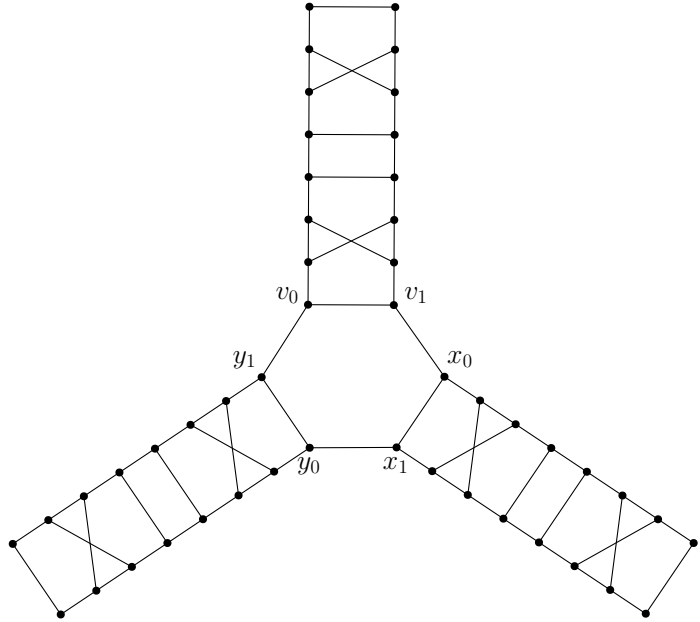}
\caption{The graph $H''$}
\label{figH''}
\end{figure}

\begin{theorem}
If a subcubic graph $G$ contains $H''$ as a subgraph, then $G$ has no good coloring.
\end{theorem}

\begin{proof}
Let $C = v_0v_1x_0x_1y_0y_1$ denote the cycle of length 6 in $H''$ which intersects all three copies of $H'$, see Figure~\ref{figH''}. Suppose for a contradiction that $G$ has a good coloring. By (1), not all vertices of $C$ are colored blue. By symmetry, we may assume that $v_0$ is colored red. By Lemma~\ref{lemmaH'}, $v_1$ is colored blue and $y_1$ is colored red. Now $y_0$ is colored blue by Lemma~\ref{lemmaH'}. By (1), not both $x_0$ and $x_1$ can be coloured blue. By symmetry, we may assume that $x_0$ is colored red. By Lemma~\ref{lemmaH'} the neighbor of $x_0$ in $G-H'$ is colored red, but $v_1$ is colored blue, a contradiction.
\end{proof}

Note that construction of $H''$ can be easily generalized. An analogous argument yields that any graph formed by gluing odd number of copies of $H'$ into a cycle as in $H''$ cannot appear as a subgraph of a subcubic graph with a good coloring.

The smallest 3-connected cubic graph containing $H''$ can be obtained from $H''$ by adding three edges joining the vertices of degree 2. However, there are many ways how to construct 3-connected cubic graphs containing $H''$ as a subgraph. For example, let $G$ be any 3-connected cubic graph containing an induced 6-cycle $C$. Since $H''$ contains precisely six vertices of degree 2, it is possible to replace $C$ by a copy of $H''$ so that the resulting graph is again 3-connected and cubic, which implies the following.

\begin{corollary}
There is an infinite family of 3-connected cubic graphs having no good coloring.
\end{corollary}

Finally, let us note that $H''$ is a 2-connected planar graph. Using $H''$, it is easy to construct an infinite family of 2-connected cubic planar graphs admitting no good coloring. We do not know if the 3-prism is the only 3-connected cubic planar graph admitting no good coloring.

\begin{ack} The research was initiated at the Structural graph theory workshop at Gułtowy,
24-28 June 2019 sponsored by ERC Starting Grant "CUTACOMBS
Cuts and decompositions: algorithms and combinatorial properties", grant agreement No 714704. The first author is supported by the Danish research council under grant number DFF-7014-00037B. The second author was supported by the grant no. 19-04113Y of the Czech Science Foundation (GAČR) and the Center for Foundations of Modern Computer Science (Charles Univ. project UNCE/SCI/004). The third author was supported by the Danish Council for Independent Research, Natural Sciences, grant DFF-8021-00249, AlgoGraph. 
\end{ack}

\bibliographystyle{abbrv}
\bibliography{bibliography}

\end{document}